\documentclass[11pt]{article}
\usepackage{amsmath,amssymb,amsthm,amscd}

\newtheorem{theorem}{Theorem}[section]
\newtheorem{lemma}[theorem]{Lemma}
\newtheorem{proposition}[theorem]{Proposition}
\newtheorem{corollary}[theorem]{Corollary}

\theoremstyle{plain}

\theoremstyle{definition}

\newtheorem{remark}[theorem]{Remark}

\numberwithin{equation}{section}

\renewcommand{\theenumi}{(\roman{enumi})}
\renewcommand{\labelenumi}{\textup{(\theenumi)}}

\title{On strong extension groups of  Cuntz--Krieger algebras \\
}
\author{Kengo Matsumoto \\
Department of Mathematics \\
Joetsu University of Education \\
Joetsu, 943-8512, Japan
}

\begin{document}


\maketitle

\date{}

\def\det{{{\operatorname{det}}}}

\begin{abstract}
In this paper, we study the strong extension groups
of Cuntz--Krieger algebras, and present a formula to compute the groups.
We also detect the position of the Toeplitz extension of a Cuntz--Krieger algebra
in the strong extension group and in the weak extension group
to see that the weak extension group with the position of the Toeplitz extension 
 is a complete invariant of the isomorphism class of the Cuntz--Krieger algebra 
 associated with its transposed matrix.  
\end{abstract}





\newcommand{\Ker}{\operatorname{Ker}}
\newcommand{\sgn}{\operatorname{sgn}}
\newcommand{\Ad}{\operatorname{Ad}}
\newcommand{\ad}{\operatorname{ad}}
\newcommand{\orb}{\operatorname{orb}}

\def\Re{{\operatorname{Re}}}
\def\det{{{\operatorname{det}}}}
\newcommand{\K}{\mathbb{K}}

\newcommand{\N}{\mathbb{N}}
\newcommand{\C}{\mathbb{C}}
\newcommand{\R}{\mathbb{R}}
\newcommand{\Rp}{{\mathbb{R}}^*_+}
\newcommand{\T}{\mathcal{T}}

\newcommand{\Z}{\mathbb{Z}}
\newcommand{\Zp}{{\mathbb{Z}}_+}
\def\AF{{{\operatorname{AF}}}}

\def\Ext{{{\operatorname{Ext}}}}
\def\Exts{{{\operatorname{Ext}_{\operatorname{s}}}}}
\def\Extw{{{\operatorname{Ext}_{\operatorname{w}}}}}
\def\Ext{{{\operatorname{Ext}}}}

\def\OA{{{\mathcal{O}}_A}}
\def\ON{{{\mathcal{O}}_N}}

\def\TA{{{\mathcal{T}}_A}}
\def\TN{{{\mathcal{T}}_N}}

\def\A{{\mathcal{A}}}
\def\OaA{{{\mathcal{O}}^a_A}}
\def\OB{{{\mathcal{O}}_B}}
\def\OTA{{{\mathcal{O}}_{\tilde{A}}}}
\def\F{{\mathcal{F}}}
\def\G{{\mathcal{G}}}
\def\FA{{{\mathcal{F}}_A}}
\def\PA{{{\mathcal{P}}_A}}
\def\C{{\mathbb{C}}}

 \def\U{{\mathcal{U}}}
\def\OF{{{\mathcal{O}}_F}}
\def\DF{{{\mathcal{D}}_F}}
\def\FB{{{\mathcal{F}}_B}}
\def\DA{{{\mathcal{D}}_A}}
\def\DB{{{\mathcal{D}}_B}}
\def\DZ{{{\mathcal{D}}_Z}}

\def\Ext{{{\operatorname{Ext}}}}
\def\Max{{{\operatorname{Max}}}}
\def\Max{{{\operatorname{Max}}}}
\def\max{{{\operatorname{max}}}}
\def\KMS{{{\operatorname{KMS}}}}
\def\Per{{{\operatorname{Per}}}}
\def\Out{{{\operatorname{Out}}}}
\def\Aut{{{\operatorname{Aut}}}}
\def\Ad{{{\operatorname{Ad}}}}
\def\Inn{{{\operatorname{Inn}}}}
\def\Int{{{\operatorname{Int}}}}
\def\det{{{\operatorname{det}}}}
\def\exp{{{\operatorname{exp}}}}
\def\nep{{{\operatorname{nep}}}}
\def\sgn{{{\operatorname{sign}}}}
\def\cobdy{{{\operatorname{cobdy}}}}
\def\Ker{{{\operatorname{Ker}}}}
\def\Coker{{{\operatorname{Coker}}}}
\def\Im{{\operatorname{Im}}}

\def\ind{{{\operatorname{ind}}}}
\def\Ind{{{\operatorname{Ind}}}}
\def\id{{{\operatorname{id}}}}
\def\supp{{{\operatorname{supp}}}}
\def\co{{{\operatorname{co}}}}
\def\scoe{{{\operatorname{scoe}}}}
\def\coe{{{\operatorname{coe}}}}
\def\I{{\mathcal{I}}}
\def\Span{{{\operatorname{Span}}}}
\def\event{{{\operatorname{event}}}}
\def\S{\mathcal{S}}

\def\coe{{{\operatorname{coe}}}}
\def\scoe{{{\operatorname{scoe}}}}
\def\uoe{{{\operatorname{uoe}}}}
\def\ucoe{{{\operatorname{ucoe}}}}
\def\event{{{\operatorname{event}}}}

\section{Preliminary}

There are several kinds of extension groups $\Ext_*(\A)$ for a $C^*$-algebra $\A$.
Among them two extension groups 
$\Extw(\A)$ and $\Exts(\A)$ for a unital nuclear separable $C^*$-algebra $\A$ 
have been studying in many papers 
(see \cite{Blackadar}, \cite{BDF}, \cite{Cuntz80}, \cite{CK}, \cite{Douglas},  \cite{HR},
\cite{Kasparov81}, \cite{PS}, \cite{PP}, \cite{PV},  etc.).
In this paper, we study the strong extension groups $\Exts(\OA)$
of Cuntz--Krieger algebras $\OA$, 
and present a formula to compute the groups.
We also detect the position of the Toeplitz extension $\TA$ of a Cuntz--Krieger algebra
$\OA$ in the weak extension group $\Extw(\OA)$ to show that it is a complete invariant of the isomorphism
class of the Cuntz--Krieger algebra $\mathcal{O}_{A^t}$ for the transposed matrix
$A^t$ of $A$ by using R{\o}rdam's classification result.

In what follows, $H$ stands for a separable infinite dimensional Hilbert space.
Let us denote by $K(H)$ the $C^*$-algebra of compact operators on $H$. 
It is a closed two-sided ideal of the $C^*$-algebra $B(H)$ of bounded linear operators on $H$.
The quotient $C^*$-algebra $B(H)/K(H)$ is called  the Calkin algebra,
denoted by $Q(H)$.
The quotient map $B(H) \longrightarrow Q(H)$ is denoted by $\pi$.

Let $\A$ be a unital separable $C^*$-algebra.
Throughout the paper,
 a unital $*$-monomorphism 
 $\tau:\A\longrightarrow Q(H)$ is called an extension.
Two extensions 
 $\tau_1, \tau_2: \A \longrightarrow Q(H)$  
are said to be {\it strongly equivalent}\/, written $\tau_1\underset{s}{\sim}\tau_2$,
if there exists a unitary $U\in B(H)$ on $H$ such that 
$\tau_1(a) = \pi(U)\tau_2(a)\pi(U^*)$  in $Q(H)$ for all $a \in \A$. 
They are said to be
{\it weakly equivalent}\/, written $\tau_1\underset{w}{\sim}\tau_2$,
 if there exists a unitary $u\in Q(H)$ such that 
$\tau_1(a) = u \tau_2(a) u^*$  in $Q(H)$ for all $a \in \A$. 
The strong equivalence class of an extensin $\tau:\A\longrightarrow Q(H)$
is denoted by  $[\tau]_s$, and similarly 
the weak equivalence class is denoted by $[\tau]_w$.
We note that weakly equivalent extensions are strongly equivalent if 
one may take a unitary $u \in Q(H)$ of Fredholm index zero such that 
$\tau_1(a) = u \tau_2(a) u^*$  in $Q(H)$ for all $a \in \A$. 
An extension $\tau:\A\longrightarrow Q(H)$ is said to be {\it trivial}\/ 
if there exists a unital $*$-monomorphism $\rho:\A\longrightarrow B(H)$ 
such that $\tau = \pi\circ\rho$.
We regard $Q(H) \oplus Q(H) \subset Q(H\oplus H)$
in a natural way and identify $H\oplus H$ with $H$, 
so that $Q(H) \oplus Q(H) \subset Q(H)$.
The sum of extensions 
 $\tau_1, \tau_2: \A \longrightarrow Q(H)$  
are defined by
$$
(\tau_1+ \tau_2)(a) =  
\tau_1(a) \oplus \tau_2(a)
\in Q(H) \oplus Q(H) \subset Q(H), \qquad a \in \A
$$
that gives rise to an extension 
$\tau_1\oplus \tau_2 : \A\longrightarrow Q(H).$
Let us denote by
$\Exts(\A)$ the set of srong equivalence classes  of extensions.
Similarly  the set of weak equivalence classes is denoted by $\Extw(\A)$. 
Both $\Exts(\A)$ and $\Extw(\A)$ have commutative semigroup structure 
by the above sums.
There is a canonical surjective homomorphism 
$q_\A: \Exts(\A) \longrightarrow \Extw(\A)$ 
of commutative semigroups
defined by $q_\A([\tau]_s) = [\tau]_w$.

By virtue of Voiculescu's theorem in \cite{Vo},
the following basic lemma holds:
\begin{lemma}[{\cite{Vo}}] \label{lem:1}
Let $\A$ be a unital separable $C^*$-algebra.
For any two trivial extensions $\tau_1, \tau_2:\A\longrightarrow Q(H)$,
there exists a unitary $U\in B(H)$ such that $\tau_2 = \Ad(\pi(U))\circ \tau_1$,
that is, $\tau_1\underset{s}{\sim}\tau_2$.  
The strong (resp. weak) equivalence class of  a trivial extension is the neutral element
of $\Exts(\A)$ (resp. $\Extw(\A)$).  
\end{lemma}
Choi--Effros in \cite{CE} (cf. \cite{Arveson}) proved that 
if $\A$ is nuclear, the semigoups 
$\Exts(\A), \Extw(\A)$ become groups, 
that is, any element has its inverse.
The following lemma is seen in \cite{PP}.
\begin{lemma}\label{lem:2}
Let $\A$ be a unital separable nuclear $C^*$-algebra.
For $m \in \Z$, take a unitary $u_m\in Q(H)$ of Fredholm index $m$.
Take a trivial extension $\tau:\A\longrightarrow Q(H)$.
Consider the extension $\sigma_m = \Ad(u_m)\circ \tau: \A\longrightarrow Q(H)$.
Then the map $\iota_\A: m\in \Z \longrightarrow [\sigma_m] \in \Exts(\A)$
gives rise to a homomorphism of groups such that 
the sequence 
\begin{equation}\label{eq:2.1iota}
\begin{CD}
\Z @>\iota_\A>> \Exts(\A)
@>q_\A >> \Extw(\A).
\end{CD}
\end{equation}
is exact at the middle, that is, $\iota_A(\Z) = \Ker(q_\A),$ 
so that 
\begin{equation*}
\Exts(\A)/ \iota_\A(\Z) \cong \Extw(\A)
\end{equation*} 
\end{lemma}
The groups $\Exts(\A)$ and $\Extw(\A)$ for a unital separable nuclear $C^*$-algebra $\A$
are called the {\it strong extension group}\/ for $\A$
and the {\it weak extension group}\/ for $\A$, respectively.

Let $e \in Q(H), E \in B(H)$ be projections such that 
$e = \pi(E)$.
For an element $x \in Q(H)$ such that $e x e \in eQ(H) e$ is invertible in $e Q(H) e$, 
one may denote by  
$\ind_ex$ the Fredholm index $\ind_EX$ for $X \in B(EH)$ satisfying 
$x =\pi(X)$.  
As the Fredholm index is invariant under compact perturbations,
the integer   
$\ind_ex$
does not depend on the choice of $E$ and $X$ as long as
$e=\pi(E), x = \pi(X)$.
The following lemma is well-known (cf. \cite[Lemma 5.1]{CK}).
\begin{lemma}\label{lem:3}
Let $e,f \in Q(H)$ be projections.
Suppose that $x\in Q(H)$ commutes with $e$ and $f$, and 
$exe, fxf$ are invertible in $e Q(H) e$ and
 $f Q(H) f,$
 respectively. 
 \begin{enumerate}
\renewcommand{\theenumi}{(\roman{enumi})}
\renewcommand{\labelenumi}{\textup{\theenumi}}
\item  If $e f =0$, then $\ind_{e + f}x = \ind_ex + \ind_fx.$ 
\item If $x,y \in eQ(H)e$ are both invertible in  $eQ(H)e$,
then $\ind_e xy = \ind_ex + \ind_ey.$ 
\end{enumerate}
\end{lemma}

\section{Ext-groups for Cuntz--Krieger algebras}
Let $A =[A(i,j)]_{i,j=1}^N$ be an irreducible non permutation matrix 
with entries in $\{0,1\}$ with $N>1$.
The Cuntz--Krieger algebra $\OA$ is defined to be the universal $C^*$-algebra generated by 
$N$ partial isometries $S_1,\dots, S_N$ subject to the operator relations:
\begin{equation}\label{eq:CK}
\sum_{j=1}^N S_j S_j^* =1, 
\qquad  
S_i^* S_i = \sum_{j=1}^N A(i,j) S_j S_j^*, \,\, \, i=1,\dots N \qquad (\cite{CK}).
\end{equation}
It is a nuclear $C^*$-algebra uniquely determined by 
the operator relations \eqref{eq:CK} (see \cite{CK}). 
For the matrix $A$ with all of the entries are one's, the $C^*$-algebra
$\OA$ is called the Cuntz algebra written $\ON$ (\cite{Cuntz77}).

In \cite{CK}, Cuntz--Krieger pointed out 
the $C^*$-algebras $\OA$ are closely related to dynamical property  
of the underlying topological Markov shifts.
Among other things, they proved that the weak extension group
$\Extw(\OA)$ is isomorphic to the abelian group
$\Z^N/ ( I - A)\Z^N.$
 The group $\Z^N/ ( I - A)\Z^N$ 
is known as the Bowen--Franks group that is 
 a crucial invariant under flow equivalence class of the underlying two-sided topological Markov shift
 (see \cite{BF}).
 We note that the group $\Extw(\OA)$ was written as 
 $\Ext(\OA)$ in the Cuntz--Krieger's paper \cite{CK}. 
 For the Cuntz algebra $\ON$, both of the groups
 $\Exts(\ON)$ and $\Extw(\ON)$ had been computed 
 as  $\Z $ and $\Z/(1 -N)\Z$, respectivly  by
 Pimsner--Popa \cite{PP} and  Paschke--Salinas \cite{PS}.

  In this paper, we will compute the strong extension group
  $\Exts(\OA)$ for $\OA$ and present a formula \eqref{eq:ExtsOA}
stated in the theorem below.
 For $n = 1,\dots, N$, 
 let $R_n =[R_n(i,j)]_{i,j=1}^N$
 be the $N\times N$ matrix defined by
 \begin{equation}\label{eq:Rn}
 R_n(i,j) =
 \begin{cases}
 1& \text{ if } i=n,\\
 0 & \text{ otherwise}
 \end{cases}
 \end{equation}
 meaning that the only $n$th row is the vector $[1,\dots,1]$
 but the other rows are zero vectors. 
  The homomorphisms 
$\iota_\A:\Z\longrightarrow \Exts(\A)$
and
$q_\A:\Exts(\A)\longrightarrow \Extw(\A)$
in \eqref{eq:2.1iota}
for $\A = \OA$ are denoted by 
$\iota_A:\Z\longrightarrow\Exts(\OA)$
and
$q_A:\Exts(\OA)\longrightarrow \Extw(\OA)$,
respectively.

\begin{theorem}[{Theorem \ref{thm:2.6} and Theorem \ref{thm:3.3}}]
\label{thm:main}\hspace{5cm}
\begin{enumerate}
\renewcommand{\theenumi}{(\roman{enumi})}
\renewcommand{\labelenumi}{\textup{\theenumi}}
\item The strong extension group
  $\Exts(\OA)$ for the Cuntz--Krieger algebra $\OA$ is 
\begin{equation}\label{eq:ExtsOA}
  \Exts(\OA) = \Z^N / ( 1- \widehat{A})\Z^N
   \end{equation}
where the matrix $\widehat{A}$ is 
$\widehat{A} = A + R_1 - AR_1$.
\item 
The homomorphism
$\iota_A: \Z \longrightarrow \Exts(\OA)$
in \eqref{eq:2.1iota}
is injective if and only if 
$\det(I-A) \ne 0$.
Hence the short exact sequence 
\begin{equation}
\begin{CD}
0 @>>> \Z @>\iota_A>> \Exts(\OA)
@>q_A >> \Extw(\OA) @>>> 0
\end{CD}
\end{equation}
holds if and only if $\det(I-A) \ne 0$.
\end{enumerate} 
 \end{theorem} 
The given proof in this paper for the formula \eqref{eq:ExtsOA} presented as Theorem \ref{thm:2.6}
is basically follows the proof of \cite[Theorem 5.3]{CK} that showed the formula
$\Extw(\OA) = \Z^N/ (I -A) \Z^N$.

Among various extensions of a Cuntz--Krieger algebra $\OA$, 
there is one specific extension called the Toeplitz extension  $\sigma_{\TA}$ of $\OA$.
It arises from the short exact sequence:
\begin{equation}
\begin{CD}
0 @>>> K(H_A) @>\iota>> \TA
@>q >> \OA @>>> 0
\end{CD}
\end{equation}
of the Toeplitz algebra $\TA$ on the sub Fock space $H_A$ (cf.  \cite{EFW}, \cite{Ev}).
We will detect the positions of the Toeplitz extension $\sigma_{\TA}$ 
of $\OA$ in the strong extension group $\Exts(\OA)$ 
and in the weak extension group $\Extw(\OA)$ (Theorem \ref{thm:4.4}).
As a result, we will know that the group $\Extw(\OA)$ with
the position $[\TA]_w$ of the Toeplitz extension $\sigma_{\TA}$
in $\Extw(\OA)$ is a complete invariant of the isomorphism
class of the Cuntz--Krieger algebra $\mathcal{O}_{A^t}$ for its transposed matrix
$A^t$ of $A$ by using R{\o}rdam's classification result (Corollary \ref{cor:4.5}).

Let us denote by $P_i$ the projection $S_i S_i^*$.
Let $\sigma :\OA \longrightarrow Q(H)$ be an extension.
Put
$e_i = \sigma(P_i)$.
There exists a trivial extension $\tau:\OA\longrightarrow Q(H)$ such that 
$\tau(P_i) = \sigma(P_i), i=1,\dots,N$.
As the partial isometry 
$\sigma(S_i) \tau(S_i^*)$
commutes with $e_i$,
$e_i\sigma(S_i) \tau(S_i^*)e_i$
becomes a unitary in 
$e_i Q(H) e_i$.
One may define 
$\ind_{e_i}\sigma(S_i) \tau(S_i^*)$
denoted by
$d_i(\sigma,\tau)$, that is
\begin{equation*}
d_i(\sigma,\tau) =\ind_{e_i}\sigma(S_i) \tau(S_i^*),
\qquad i=1,\dots,N.
\end{equation*}
The proof of \cite[Proposition 5.2]{CK} describes 
the following lemma.
We give its proof for the sake of completeness.
\begin{lemma}[{cf. \cite[Proposition 5.2]{CK}}] \label{lem:2.2}
Let $\sigma :\OA \longrightarrow Q(H)$ be an extension.
Put
$e_i = \sigma(P_i)$.
Let $\tau_1, \tau_2 :\OA\longrightarrow Q(H)$ be trivial extensions  
such that 
$\tau_j(P_i) = \sigma(P_i), j=1,2, \, \,  i=1,\dots,N$.
Then there exists a vector $[k_i]_{i=1}^N \in \Z^N$
such that 
\begin{enumerate}
\renewcommand{\theenumi}{(\arabic{enumi})}
\renewcommand{\labelenumi}{\textup{\theenumi}}
\item
$d_i(\sigma,\tau_2) 
=
d_i(\sigma,\tau_1) -k_i + \sum_{j=1}^N A(i,j) k_j, $
\item
$\sum_{i=1}^N k_i =0$.
\end{enumerate}
\end{lemma}
\begin{proof}
By Lemma \ref{lem:1},
one may find a unitary $U \in B(H)$ such that 
$\tau_2(x) = \pi(U) \tau_1(x) \pi(U^*), x \in \OA$. 
Put
$u =\pi(U)\in Q(H).$
Since 
$$
(e_i u e_i)(e_i u e_i)^*
= 
\tau_2(P_i) \pi(U) \tau_1(P_i) \pi(U^*) \tau_2(P_i) 
= 
\tau_2(P_i) \tau_2(P_i)\tau_2(P_i) 
= e_i 
$$ 
and similarly 
$
(e_i u e_i)^*(e_i u e_i) = e_i,
$
we see that 
$
e_i u e_i
$ is a unitary in $e_iQ(H) e_i.$
By putting
$ k_i = \ind_{e_i}u,$
the equality
\begin{equation}\label{eq:21}
d_i(\sigma,\tau_2) 
=
d_i(\sigma,\tau_1) -k_i + \sum_{j=1}^N A(i,j) k_j 
\end{equation}
holds, following the proof of \cite[Proposition 5.2]{CK}.
We in fact  see that  
\begin{align*}
d_i(\sigma,\tau_2)
= & \ind_{e_i}\sigma(S_i)\tau_2(S_i^*) \\
= & \ind_{e_i}\sigma(S_i)\sigma(S_i^* S_i) u\tau_1(S_i^*S_i)\tau_1(S_i^*)u^* \\
= & \ind_{e_i}\sigma(S_i)\tau_1(S_i^* S_i) u\tau_1(S_i^*S_i)\tau_1(S_i^*)\tau_1(S_i S_i^*) u^* \\
= & \ind_{e_i}\sigma(S_i)\tau_1(S_i^*)
\left( \tau_1( S_i)\sum_{j=1}^N A(i,j) u\tau_1(S_j S_j^*)\tau_1(S_i^*)\right) e_i u^* \\ 
= & \ind_{e_i}\sigma(S_i)\tau_1(S_i^*)
\left( \tau_1( S_i)\sum_{j=1}^N A(i,j) e_j u e_j \tau_1(S_i^*)\right) e_i u^* e_i \\
= & \ind_{e_i}\sigma(S_i)\tau_1(S_i^*)
  + \ind_{e_i} \left(\tau_1( S_i)\sum_{j=1}^N A(i,j) e_j u e_j \tau_1(S_i^*)\right)
  + \ind_{e_i}  u^*  \\ 
= & d_i(\sigma, \tau_1) 
+ \sum_{j=1}^N A(i,j) \ind_{e_i} \tau_1( S_i) e_j u e_j \tau_1(S_i^*)
     - k_i. 
\end{align*}
As 
$ \ind_{e_i} \tau_1( S_i) e_j u e_j \tau_1(S_i^*) = \ind_{e_j} u = k_j$
whenever $A(i,j) =1$,
we obtain the equality \eqref{eq:21}.

Since $u = \pi(U)$ for some unitary $U$ on $H$,
Lemma \ref{lem:3} tells us 
\begin{equation*}
\sum_{i=1}^N k_i
 =  \sum_{i=1}^N \ind_{e_i}u 
 = \ind_{\sum_{i=1}^N e_i} u 
=\ind U =0.
\end{equation*}
\end{proof}
Define a subgroup
$\Im(1-A)_0 \subset \Z^N$ by setting
\begin{equation*}
\Im(I-A)_0  = 
\{ (I -A)[k_i]_{i=1}^N \in \Z^N
\mid
[k_i]_{i=1}^N
\in \Z^N \text{ with }
\sum_{i=1}^N k_i =0
\}.
\end{equation*}
We thus see that 
an extension $\sigma :\OA \longrightarrow Q(H)$
defines an element of 
$\Z^N / {\Im(I-A)_0}$ in a unique way by 
\begin{equation*}
d_\sigma: = [d_i(\sigma, \tau)]_{i=1}^N \in \Z^N / {\Im(I-A)_0}
\end{equation*}
for a trivial extension $\tau: \OA\longrightarrow Q(H)$
satisfying
$\tau(P_i) =\sigma(P_i), i=1,\dots,N$. 
 
\begin{lemma}\label{lem:2.3}
Let $\sigma_1, \sigma_2:\OA\longrightarrow Q(H)$ be extensions.
If $\sigma_1\underset{s}{\sim} \sigma_2$, then $d_{\sigma_1} = d_{\sigma_2}$ in 
$\Z^N / {\Im(I-A)_0}$.
\end{lemma}
\begin{proof}
Assume that $\sigma_1\underset{s}{\sim} \sigma_2$ 
so that, by Lemma \ref{lem:1},
one may find a unitary $V$ in $B(H)$ such that 
 $\sigma_2 = \Ad(\pi(V))\circ\sigma_1$.
Put $e^1_i = \sigma_1(P_i),e^2_i = \sigma_2(P_i), i=1,\dots,N$ 
and $v =\pi(V)\in Q(H)$, and hence $ve^1_i v^* = e^2_i$.
 Take a trivial extension $\tau_1:\OA\longrightarrow Q(H)$ such that 
 $\tau_1(P_i) = e_i^1, i=1,\dots,N.$
We set $\tau_2 = \Ad(v)\circ\tau_1$
so that 
$
\tau_2(P_i) 
= v \tau_1(P_i)v^* 
= e_i^2. 
$
We then have
\begin{equation*}
d_i(\sigma_2, \tau_2)
= \ind_{e_i^2}\sigma_2(S_i)\tau_2(S_i^*) 
= \ind_{ve_i^1v^*} v\sigma_1(S_i)\tau_1(S_i^*)v^* 
= d_i(\sigma_1, \tau_1).
\end{equation*}
\end{proof}
Let us define $d_s: \Exts(\OA) \longrightarrow \Z^N / {\Im(I-A)_0}$
by setting 
\begin{equation*}
d_s([\sigma]_s) = [[d_i(\sigma,\tau)]_{i=1}^N]  \in \Z^N / {\Im(I-A)_0}
\end{equation*}
for a trivial extension $\tau:\OA\longrightarrow Q(H)$
satisfying $\tau(P_i) = \sigma(P_i), i=1,\dots,N$.

\begin{proposition}\label{prop:2.4}
$d_s: \Exts(\OA) \longrightarrow \Z^N/\Im(I-A)_0 $
is an isomorphism of groups.
\end{proposition}
\begin{proof}
It is obvious that $d_s: \Exts(\OA) \longrightarrow \Z^N/\Im(I-A)_0$
is a homomorphism of groups. 
It remains to show that $d_s$ is bijective.
We will first show that 
$d_s$ is injective.
Let $\sigma:\OA\longrightarrow Q(H)$ be an extension such that 
$d_s([\sigma]_s) =0$ in $\Z^N/\Im(I-A)_0.$ 
Take a trivial extension $\tau$ such that 
$\tau(P_i) = \sigma(P_i), i=1,\dots,N$.
Put $d_i = d_i(\sigma,\tau)\in \Z$.
Let $\rho_\tau:\OA\longrightarrow B(H)$ 
be a unital $*$-monomorphism such that 
$\tau = \pi\circ \rho_\tau$.  
By the assumption, there exists $[k_i]_{i=1}^N \in \Z^N$
such that 
\begin{equation*}
[d_i]_{i=1}^N = (I -A) [k_i]_{i=1}^N, \qquad \sum_{i=1}^N k_i =0.
\end{equation*}
Put $e_i = \tau(P_i)$ and $E_i = \rho_\tau(P_i)$ so that 
$\pi(E_i) = e_i$. 
 Take an isometry or coisometry $V_i \in B(E_iH)$ such that 
  $\ind(V_i) = - k_i$.
Put $V = \sum_{i=1}^N V_i \in B(H)$ and $v =\pi(V)$.
Since $v$ is a unitary in $Q(H)$ such that
$\ind(v) = \sum_{i=1}^N\ind_{E_i}(V_i) = - \sum_{i=1}^N k_i =0,$ 
one may take a unitary
$U $ in $B(H)$ such that $v = \pi(U)$.
By following the proof of \cite[Theorem 5.3]{CK},
we have
\begin{align*}
  & \ind_{e_i} \pi(U) \sigma(S_i) \pi(U^*) \tau(S_i^*) \\
= &
\ind_{e_i} \pi(V_i) \sigma(S_i)\sigma(S_i^* S_i) \pi(\sum_{n=1}^N V_n^*) \tau(S_i^*) \\
= &
\ind_{e_i} \pi(V_i) \sigma(S_i) \left(\sum_{j=1}^N A(i,j) \pi(E_j)\right)
\pi(\sum_{n=1}^N E_n V_n^*) \tau(S_i^*) \\
= &
\ind_{e_i} \pi(V_i) \sigma(S_i)\sigma(S_i^* S_i)\pi( \sum_{j=1}^N A(i,j) V_j^*) \tau(S_i^*) \\
= &
\ind_{e_i} \pi(V_i) \sigma(S_i)\tau(S_i^*) 
\left( \tau( S_i)\pi( \sum_{j=1}^N A(i,j) V_j^*) \tau(S_i^*)\right) \\
= &
\ind_{e_i} \pi(V_i)  + \ind_{e_i}\sigma(S_i)\tau(S_i^*) 
+ \ind_{e_i} \tau( S_i)\pi( \sum_{j=1}^N A(i,j) V_j^*) \tau(S_i^*) \\
= &
 -k_i + d_i +  \sum_{j=1}^N A(i,j) \ind_{e_i} \tau( S_i)\pi( V_j^*) \tau(S_i^*). 
\end{align*}
Since
$\ind_{e_i} \tau( S_i)\pi( V_j^*) \tau(S_i^*) = \ind_{e_j} \pi( V_j^*) = k_j$
whenever $A(i,j) =1$,
we have
\begin{equation*}
   \ind_{e_i} \pi(U) \sigma(S_i) \pi(U^*) \tau(S_i^*) 
=  -k_i + d_i +  \sum_{j=1}^N A(i,j)k_j = 0
\end{equation*}
so that there exists  a unitary $W_i \in B(E_iH)$ on $E_iH$ such that 
\begin{equation*}
   \pi(U) \sigma(S_i) \pi(U^*) \tau(S_i^*) 
=  \pi(W_i), \qquad i=1,\dots, N.
\end{equation*}
By putting $T_i = W_i \rho_\tau(S_i), i=1,\dots,N$, we have
\begin{equation*}
\sum_{j=1}^N T_j T_j^* 
= \sum_{j=1}^N  W_j \rho_\tau(S_j)\rho_\tau(S_j^*)W_j^*  
= \sum_{j=1}^N  W_j W_j^* = \sum_{j=1}^N  E_j = 1,
\end{equation*}
 and
 \begin{equation*}
T_i^* T_i 
=  \rho_\tau(S_i^*)W_i^* W_i \rho_\tau(S_i) 
 =  \rho_\tau(S_i^*)\rho_\tau(S_i S_i^*) \rho_\tau(S_i) 
 =  \sum_{j=1}^N A(i,j) \rho_\tau(S_j S_j^*).
 \end{equation*}
 As  $\rho_\tau(S_j S_j^*)= T_j T_j^*,$
 we see that 
 $ T_i^* T_i 
  =  \sum_{j=1}^N A(i,j)  T_j T_j^*.
  $
 Define $\rho_\sigma(S_i) = T_i \in B(H), i=1,\dots,N$ 
 so that 
 $\rho_\sigma:\OA\longrightarrow Q(H)$ is a unital $*$-monomorphism
 such that 
 \begin{align*}
 \pi\circ \rho_\sigma(S_i)
 = & \pi(W_i \rho_\tau(S_i)) \\
  = &  \pi(U) \sigma(S_i) \pi(U^*) \tau(S_i)\tau(S_i^*)  \\
 = &  \pi(U) \sigma(S_i) \tau(S_i S_i^*)\pi(U^*)   \\
 = &  \pi(U) \sigma(S_i) \pi(U^*).   
  \end{align*}
Hence we have
$\Ad(\pi(U))\circ \sigma = \pi\circ \rho_\sigma.$
This shows that $\sigma$ is strongly equivalent to 
the trivial extension
$\pi\circ \rho_\sigma$
proving $[\sigma]_s =0$ in $\Exts(\OA)$.

We will next show that $d_s$ is surjective.
We will show that there exist an extension
$\sigma:\OA\longrightarrow Q(H)$ 
and a trivial extension
$\tau:\OA\longrightarrow Q(H)$ 
such that $\tau(P_i) = \sigma(P_i)$ denoted by $e_i$
and
\begin{equation}\label{eq:2.2.4}
\ind_{e_i}\sigma(S_i)\tau(S_i^*) =
\begin{cases}
-1 & \text{ if } i=1,\\
0 & \text{ otherwise. }
\end{cases}
\end{equation}
Decompose the Hilbert space $H$
as $H =H_1\oplus\cdots\oplus H_N$ such that 
$\dim H_i = \dim H,i=1,\dots,N$. 
Take a nonzero vector $v_1 \in H_1$ 
and put its orthogonal complement 
$H_1^0= \{\mathbb{C} v_1\}^\perp \cap H_1$ in $H_1$.
Let $E_i$ be the orthogonal projection onto $H_i, i=1,\dots,N$.
The orthogonal projection onto $H_1^0$ is denoted by $E_1^0$,
so that 
$\sum_{i=1}^N E_i =1$ and $E_1 - E_1^0$ is the projection onto 
$\mathbb{C} v_1$.
Take partial isometries $T_1,\dots,T_N$ and $V_1,\dots,V_N$ on $H$
such that 
\begin{gather*}
T_1 T_1^* = E_1^0, \quad T_i T_i^* = E_i,\,\,  i=2,\dots,N,
\qquad
V_i V_i^* = E_i,\,\,  i=1,\dots,N \quad \text{ and } \\
T_i^* T_i = V_i^* V_i =\sum_{j=1}^N A(i,j) E_j, \quad i=1,\dots,N.
\end{gather*}
We know that 
\begin{gather*}
\pi(T_i) \pi(T_i)^* = \pi(V_i)\pi( V_i)^* =\pi(E_i), \qquad
\sum_{i=1}^N\pi(E_i) =1\quad \text{ and } \\ 
\pi(T_i)^* \pi(T_i) = \pi(V_i)^* \pi(V_i) = \sum_{j=1}^N A(i,j) \pi(E_j),\quad
i=1,\dots,N. 
\end{gather*}
By setting
$\sigma(S_i) = \pi(T_i), \tau(S_i) = \pi(V_i), i=1,\dots,N,$
we have extensions 
$\sigma, \tau:\OA\longrightarrow Q(H)$ such that 
$\tau$ is a trivial extension. 
Put $e_i = \pi(E_i), i=1,\dots,N.$
Since
$\sigma(S_i)\tau(S_i^*) = \pi(T_i V_i^*), i=1,\dots,N,$
we have
$\ind_{e_i}\sigma(S_i)\tau(S_i^*) = \ind_{E_i}T_iV_i^* 
$
so that 
the equality
\eqref{eq:2.2.4}
holds.
Therefore we have
$d_s([\sigma]_s) = [(-1,0,\dots,0)] 
$
in $\Z^N/ \Im(I-A)_0$.
One may show that  
$d_s:\Exts(\OA)\longrightarrow\Z^N/ \Im(I-A)_0$
is surjective by a similar fashion.
\end{proof}
Recall that the $N\times N$ matrix $R_n$ for $n=1,\dots,N$ is defined in \eqref{eq:Rn}.
\begin{lemma}\label{lem:2.5}
For $n=1,2,\dots,N$,
put $\widehat{A}_n = A + R_n - A R_n.$
Then we have
\begin{equation}
\Im(I- A)_0 = (I - \widehat{A}_n)\Z^N.
\end{equation}
In particular  for $n=1$, 
we put $\widehat{A} = \widehat{A}_1$ 
so that we have
$\Im(I- A)_0 = (I - \widehat{A})\Z^N.$
 \end{lemma}
\begin{proof}
As 
$\Im(I-A)_0 = \{ (I -A) [k_i]_{i=1}^N \mid \sum_{i=1}^N k _i =0 \},$
a vector $[k_i]_{i=1}^N  \in \Z^N$ 
satisfies $\sum_{i=1}^N k _i =0$
if and only if $[k_i]_{i=1}^N  = ( I - R_n) [k_i]_{i=1}^N $.
Hence we have
$$
\Im(I -A)_0 = (I -A) ( I - R_n) \Z^N.
 $$
Since 
$(I -A ( I - R_n) = I - \widehat{A}_n$, 
we have 
$\Im(I- A)_0 = \Im(I - \widehat{A}_n)\Z^N$.
\end{proof}
Therefore we reach the following theorem.
\begin{theorem}\label{thm:2.6}
$\Exts(\OA) \cong \Z^N/( I - \widehat{A})\Z^N,$
where $\widehat{A} = A + R_1 - AR_1.$
\end{theorem}

\section{The homomorphism $\iota_A: \Z \longrightarrow \Exts(\OA)$}
For $m \in \Z$, take $k_1,\dots, k_N \in \Z$ such that 
$ m = \sum_{j=1}^N k_j$.
Take trivial extensions $ \tau, \tau' : \OA\longrightarrow Q(H)$ such that 
$\tau(P_i)= \tau'(P_i)$ denoted by $e_i, i=1,\dots,N$.
Let $\rho_\tau, \rho_{\tau'} :\OA\longrightarrow B(H)$ 
be unital $*$-monomorphisms such that 
$\tau = \pi\circ\rho_\tau, 
\tau' = \pi\circ\rho_{\tau'},
$
respectivley.
Put $ E_i = \rho_\tau(P_i)$ so that 
$\pi(E_i) = e_i$.
Take an isometry or coisometry $V_i \in B(E_iH)$ 
such that 
$\ind_{E_i} V_i = k_i$ and put 
$ V = \sum_{i=1}^N V_i \in B(H)$.
Hence we see that 
\begin{equation*}
\ind_{e_i} \pi(V) = k_i, \qquad i=1,\dots,N.
\end{equation*}
Recall that the  extension $\sigma_m: \OA\longrightarrow Q(H)$ 
is defined by setting
$\sigma_m = \Ad(\pi(V)) \circ \tau :\OA \longrightarrow Q(H)$.
Put
$d_i = d_i(\sigma_m, \tau') = \ind_{e_i}\sigma_m(S_i)\tau'(S_i^*)$.
Then 
$d_s([\sigma_m]_s) = [(d_1,\dots,d_N)] \in \Z^N/ (I-\widehat{A})\Z^N$
does not depend on the choice of trivial extensions $\tau, \tau'$,
becuuse of Lemma \ref{lem:2.2} and Lemma \ref{lem:2.3}.
\begin{proposition}\label{prop:3.1}
Define $\hat{\iota}_A: \Z \longrightarrow  \Z^N/ (I-\widehat{A})\Z^N$
by setting
$\hat{\iota}_A(m) = [(I-A)[k_i]_{i=1}^N ]$ for 
$m = \sum_{i=1}^Nk_i.$
Then we have
\begin{enumerate}
\renewcommand{\theenumi}{(\roman{enumi})}
\renewcommand{\labelenumi}{\textup{\theenumi}}
\item
$\hat{\iota}_A(m) = [(I-A)[k_i]_{i=1}^N]$ does not depend 
on the choice of $[k_i]_{i=1}^N$ as long as $m = \sum_{i=1}^Nk_i.$
\item The diagram
\begin{equation*}
\begin{CD}
\Z @>\iota_A>> \Exts(\OA)\\
@V{=}VV   @VV{d_s}V\\
\Z @>>\hat{\iota}_A> \Z^N/ (I-\widehat{A})\Z^N
\end{CD}
\end{equation*}
is commutative, that is 
$d_s(\iota_A(m)) = \hat{\iota}_A(m),$
where $\iota_A(m) = [\sigma_m]_s.$
\item
The position $\hat{\iota}_A(1)$
 in $\Z^N/ (I-\widehat{A})\Z^N$ 
 is invariant under the isomorphism
 class of $\OA$.
\item If $\det(I- A) \ne 0$,
then we have a short exact sequence 
\begin{equation}
\begin{CD}
0 @>>> \Z @>\iota_A>> \Exts(\OA)
@>q_A >> \Extw(\OA) @>>> 0.
\end{CD}
\end{equation}
\end{enumerate}
\end{proposition}
\begin{proof}
(i)
Suppose that 
$m = \sum_{i=1}^N k_i = \sum_{i=1}^N k'_i$
 for some $k_i, k'_i \in \Z$.
Put $ l_i = k_i - k'_i$ so that 
$\sum_{i=1}^Nl_i =0$ and 
$(I-A)[k_i]_{i=1}^N - (I-A)[k'_i]_{i=1}^N  =(I-A)[l_i]_{i=1}^N \in (I -\widehat{A})\Z^N$
by Lemma \ref{lem:2.5}.  
This shows that 
$[(I-A)[k_i]_{i=1}^N] =[(I-A)[k'_i]_{i=1}^N]$ in $\Z^N/ (I -\widehat{A})\Z^N.$

(ii) Keep the notation stated before Proposition \ref{prop:3.1}.
Since 
$
d_i 
= \ind_{e_i} \sigma_m(S_i)\tau'(S_i^*)
$
does not depend on the choice of a trivial extension 
$\tau':\OA\longrightarrow Q(H)$ as long as $\tau(P_i) = \tau'(P_i)$,
we may take $\tau' $ as $\tau$.
We then have
\begin{align*}
d_i 
=& \ind_{e_i} \sigma_m(S_i)\tau(S_i^*) \\
=& \ind_{e_i} \pi(V)\tau(S_i)\pi(V^*) \tau(S_i^*) \\
=& \ind_{e_i} \pi(V)
+ \ind_{e_i} \tau(S_i)\pi(V^*) \tau(S_i^*) \\
=& k_i
+ \ind_{\tau(S_i^*P_i S_i)} \pi(V^*)  \\
=& k_i
+ \sum_{j=1}^N A(i,j) \ind_{\tau(P_j)} \pi(V^*)  \\
=& k_i -\sum_{j=1}^NA(i,j)k_j
\end{align*}
so that 
we obtain
$$
d_s(\iota_A(m)) = d_s([\sigma_m]_s) =[d_i]_{i=1}^N 
=[(I-A)[k_i]_{i=1}^N] = \hat{\iota}_A(m).
$$ 

(iii)
By the construction, the map
$\iota_\A: m \in \Z \longrightarrow [\sigma_m]_s \in \Exts(\A)$
as well as
the position 
$\iota_\A(1)$ in $\Exts(\A)$ 
is invariant under the isomorhism class of a $C^*$-algebra $\A$.
For $\A = \OA$, the assertion (ii) says that 
$$
(\Exts(\OA), \iota_A(1)) \cong (\Z^N/(I-\widehat{A})\Z^N, \hat{\iota}_A(1))
$$
so that the position of $\hat{\iota}_A(1)$ in the group
$\Z^N/(I-\widehat{A})\Z^N$ 
is invariant under the isomorphism class of $\OA$.

(iv)
Assume that $\det(I -A) \ne 0$.
Let $m \in \Z$ satisfy $\iota_A(m) =0.$
Take $k_1,\dots,k_N\in \Z$ such that 
$m = \sum_{i=1}^Nk_i$ and hence
$\hat{\iota}_A(m) = [(I-A)[k_i]_{i=1}^N].$
As
$\hat{\iota}_A(m) = d_s(\iota_A(m)) = 0,$ 
there exists $[n_i]_{i=1}^N\in \Z^N$ such that 
$\sum_{i=1}^N n_i =0$ and
$\hat{\iota}_A(m) = (I-A)[n_i]_{i=1}^N$.
We then have
$(I-A)[k_i]_{i=1}^N= (I-A)[n_i]_{i=1}^N.$
By the assumption $\det(I -A)\ne 0$,
we have
 $[n_i]_{i=1}^N= [k_i]_{i=1}^N$
 so that 
 $m = \sum_{i=1}^N n_i =0.$
\end{proof}
Since 
$I-\widehat{A} = (I -A)(I-R_1)$,
the inclusion relation
$(I-\widehat{A})\Z^N \subset  (I -A)\Z^N$
holds.
There exists a natural quotient map
$\hat{q}_A: \Z^N/(I-\widehat{A})\Z^N
\longrightarrow \Z^N/(I-A)\Z^N$.
In \cite{CK}, Cuntz--Krieger proved that 
the map
$d_w: \Extw(\OA)\longrightarrow \Z^N/(I-A)\Z^N$
defined by
$d_w([\sigma]_w) =[(d_1,\dots,d_N)] \in \Z^N/(I-A)\Z^N$
yields an isomorphism of groups.

Let us denote by
$\Ker(I -A), \Ker(I - \widehat{A})$
the subgroups of $\Z^N$ defined by
the kernels in $\Z^N$ of the matrices $I-A$ and of $I - \widehat{A},$
respectively.
Define homomorphisms of groups
$$
i_1: \Z\longrightarrow \Ker(I - \widehat{A}),\qquad
j_A: \Ker(I - \widehat{A})  \longrightarrow \Ker(I - A),\qquad
s_A: \Ker(I - A)\longrightarrow \Z
$$
by setting
$$
i_1: n \longrightarrow
\begin{bmatrix}
n\\
0\\
\vdots\\
0
\end{bmatrix},
\qquad
j_A:
[l_i]_{i=1}^N \longrightarrow
\begin{bmatrix}
- \sum_{i=2}^N l_i\\
l_2\\
\vdots\\
l_N
\end{bmatrix},
\qquad
s_A:
[l_i]_{i=1}^N \longrightarrow
\sum_{i=1}^N l_i.
$$
\begin{lemma}
We have the following long exact sequence.
\begin{equation*}
\begin{CD}
0 @>>> \Z @>{i_1}>> \Ker(I-\widehat{A}) @>j_A>> \Ker(I-A) \\
@. @. @.   @VV{s_A}V  \\
0  @<<<    \Z^N/ (I-A)\Z^N  @<\hat{q}_A<< \Z^N/ (I-\widehat{A})\Z^N @<\hat{\iota}_A<< \Z 
\end{CD}
\end{equation*}
\end{lemma}
\begin{proof}
It suffices to show the exactness at the lower right corner
\begin{equation}\label{eq:exactatlr}
\begin{CD}
\Ker(I-A) @>s_A>> \Z
@>\hat{\iota}_A >> \Z^N/ (I-\widehat{A})\Z^N.
\end{CD}
\end{equation}
Suppose that  
$m \in \Z$ satisfies $\hat{\iota}_A(m) =0.$
Take $k_1,\dots,k_N\in \Z$ such that 
$m = \sum_{i=1}^Nk_i$ and hence
$(I-A)[k_i]_{i=1}^N$
belongs to $\Im(I-A)_0.$
There exists $[n_i]_{i=1}^N\in\Z^N$ such that 
$(I-A)[k_i]_{i=1}^N= (I-A)[n_i]_{i=1}^N$
and
$\sum_{i=1}^N n_i =0.$
Put $l_i = k_i - n_i$.
Hence 
$[l_i]_{i=1}^N \in \Ker(I-A)$
and
$ \sum_{i=1}^N l_i = \sum_{i=1}^N k_i = m$
so that 
$s_A([l_i]_{i=1}^N )= m,$
proving  
$\Ker(\hat{\iota}_A) \subset s_A(\Ker(I-A))$.

Conversely, for $[l_i]_{i=1}^N\in \Ker(I-A)$, we have
$
\hat{\iota}_A(s_A([l_i]_{i=1}^N)) 
=\hat{\iota}_A(\sum_{i=1}^N l_i) 
=[(I-A)[l_i]_{i=1}^N] =0,
$
so that 
$s_A(\Ker(I-A)) \subset \Ker(\hat{\iota}_A).$
Hence the sequence  \eqref{eq:exactatlr}
is exact at the middle.
Exactness at the other places are easily seen.
\end{proof}

We thus have the following theorem.
\begin{theorem}\label{thm:3.3}
\begin{enumerate}
\renewcommand{\theenumi}{(\roman{enumi})}
\renewcommand{\labelenumi}{\textup{\theenumi}}
\item
The isomorphisms
$$
d_w : \Extw(\OA)\longrightarrow \Z^N/(I-A)\Z^N,\qquad
d_s: \Exts(\OA)\longrightarrow \Z^N/(I-\widehat{A})\Z^N
$$
of groups and a homomorphism
$\hat{\iota}_A: \Z \longrightarrow \Z^N/(I-\widehat{A})\Z^N
$
defined by 
$\hat{\iota}_A(m) = (I-A)[k_i]_{i=1}^N$ 
with
$m= \sum_{i=1}^N k_i$
yield the commutative diagrams: 
\begin{equation*}
\begin{CD}
\Z @>\iota_A>> \Exts(\OA) @>q_A>> \Extw(\OA)\\
@V{=}VV   @VV{d_s}V @VV{d_w}V \\
\Z @>>\hat{\iota}_A> \Z^N/ (I-\widehat{A})\Z^N @>>\hat{q}_A> \Z^N/ (I-A)\Z^N.
\end{CD}
\end{equation*}
\item
The pair $(\Z^N/ (I-\widehat{A})\Z^N, \hat{\iota}_A(1))$ showing 
the position $\hat{\iota}_A(1)=[(I-A)[
\begin{bmatrix}
1\\
0\\
\vdots\\
0
\end{bmatrix}
]$
in the group $\Z^N/ (I-\widehat{A})\Z^N$
is invariant under the isomorphism class of $\OA$.
\item The homomorphism
$\iota_A:\Z\longrightarrow \Exts(\OA)$ is injective if and only if 
$\det(I- A) \ne 0$. In this case,  
 we have a short exact sequence 
\begin{equation}\label{eq:thm3.33}
\begin{CD}
0 @>>> \Z @>\iota_A>> \Exts(\OA)
@>q_A >> \Extw(\OA) @>>> 0.
\end{CD}
\end{equation}
\end{enumerate}
\end{theorem}

\section{Toeplitz extension}

Among various extensions of $\OA$, 
there is a specific extension $\sigma_\TA$ of $\OA$ called
the Toeplitz extension (cf. \cite{EFW}, \cite{Ev}). 
We fix an irreducible non permutation matrix $A =[A(i,j)]_{i,j=1}^N$
with entries in $\{0,1\}$.
Let $\C^N$ be an $N$-dimensional Hilbert space with orthonormal
basis $\{\xi_1,\dots,\xi_N\}$.
Let $H_0$ be a one-dimensional Hilbert space with unit vector $v_0$.
Let $H^{\otimes n}$ be the $n$-fold tensor product 
$\C^N\otimes\cdots\otimes \C^N$.
Consider the full Fock space
$F_N =H_0 \oplus(\oplus_{n=1}^\infty H^{\otimes n})$.  
Define a sub Fock space $H_A$ to be the closed linear span of vectors
$$
\{v_0\} \cup \{\xi_{i_1}\otimes \cdots \otimes \xi_{i_n} \mid A(i_j, i_{j+1}) =1
\text{ for } j =1,\dots,n-1, \, n=1,2,\dots \}. 
$$
Define creation operators $T_i$ for $i=1,\dots,N$ on $H_A$ by
\begin{align*}
T_i v_0 = & \xi_i, \\
T_i(\xi_{i_1}\otimes\cdots \otimes \xi_{i_n})
=&
{\begin{cases}
 \xi_i\otimes \xi_{i_1}\otimes\cdots \otimes \xi_{i_n}
& \text{ if } A(i, i_1)=1,\\
0 & \text{ otherwise.}
\end{cases}
}
\end{align*}
Let us denote by $E_0$ the rank one projection  
onto the subspace $H_0$ on $H_A$.
The operators  $T_i, i=1,\dots,N$ on $H_A$ are partial isometries 
satisfying the relations
\begin{equation}
\label{eq:Toeplitz}
\sum_{j=1}^N T_j T_j^* =1-E_0, 
\qquad  
T_i^* T_i = \sum_{j=1}^N A(i,j) T_j T_j^* + E_0, \,\, \, i=1,\dots N \quad (\cite{EFW}, \cite{Ev}).
\end{equation}
The Toeplitz algebra for the matrix $A$ is defined to be the $C^*$-algebra
$C^*(T_1,\dots,T_N)$ on $H_A$ generated by 
the partial isometries $T_i, i=1,\dots,N$. 
By \eqref{eq:Toeplitz},
 we know that the correspondence 
$S_i\in \OA \longrightarrow \pi(T_i) \in Q(H_A) = B(H_A)/K(H_A)$
gives rise to a unital $*$-monomorphism, 
that is called the Toeplitz extension denoted by $\sigma_{\TA}$. 
In this section, we will detect the positions 
$d_s([\sigma_{\TA}]_s)$ in $\Exts(\OA)$
and
$d_w([\sigma_{\TA}]_w)$ in $\Extw(\OA),$
respectively.
The classes $[\sigma_{\TA}]_s$ and $[\sigma_{\TA}]_w$
are simply denoted by  $[{\TA}]_s$ and $[{\TA}]_w$,
respectively.

For $j=1,\dots,N$, let
$H_{A,j}$ be the closed linear subspace of $H_A$
spanned by the vectors 
$\{ \xi_j\otimes\eta \in H_A\mid \eta \in H_A\}$, so that 
$H_A = H_0 \oplus H_{A,1}\oplus\cdots\oplus H_{A,N}$.
Let us denote by
$E_{A,i}$ the projection on $H_A$ onto the subspace $H_{A,i}$.
We then see that $E_0 + \sum_{j=1}^N E_{A,j}=1$ and
\begin{equation}\label{eq:TCK}
T_i T_i^* =E_{A,i}, 
\qquad  
T_i^* T_i = E_0 + \sum_{j=1}^N A(i,j) E_{A,j} , \qquad i=1,\dots N.
\end{equation}

We fix $m\in \{1,\dots,N\}$ for a while.
By setting 
\begin{equation*}
H_j: = 
\begin{cases}
H_{A,j}\oplus H_0 & \text{ if } j=m,\\
H_{A,j} & \text{ if } j\ne m.
\end{cases}
\end{equation*}
we have a decomposition
$H_A = H_1\oplus\cdots \oplus H_N$
of $H_A$ depending on $m$.
Let us denote by $E_i$ the orthogonal projection on $H_A$ onto the subspace $H_i$,
so that we have $\sum_{j=1}^N E_j =1.$
Take a family of partial ismetries $V_1,\dots, V_N$ on $H_A$ satisfying the relations
\begin{equation}\label{eq:VCK}
V_i V_i^* =E_i, 
\qquad  
V_i^* V_i = \sum_{j=1}^N A(i,j) E_j, \qquad i=1,\dots N.
\end{equation}
\begin{lemma}\label{lem:4.1}
 For a fixed $m \in\{1,\dots, N\}$, we have for $i=1,\dots,N$,
 \begin{equation*}
 E_i =
 \begin{cases}
 E_{A,i} + E_0 & \text{ if } i=m,\\
E_{A,i} & \text{ if } i\ne m, 
 \end{cases}
 \qquad
 V_i^* V_i  =
 \begin{cases}
 T_i^* T_i & \text{ if } A(i,m) =1,\\
T_i^* T_i - E_0 & \text{ if } A(i,m) =0. 
 \end{cases}
 \end{equation*}
\end{lemma} 
For $i=1,\dots,N$, the operator $T_i E_0 T_i^*$ on $H_A$ 
is a rank one projection on $H_A$
onto the one-dimensional subspace spanned by the vector 
$\xi_i$.
We note that the operator $T_i V_i^*: H_i\longrightarrow H_i$
is a (not necessarily onto) partial isometry.
We then have
\begin{lemma} For $i=1,\dots,N,$ we have
$(T_i V_i^*)^* T_i V_i^* = V_i V_i^* = E_i$ and
\begin{equation}\label{eq:lem4.2}
T_i V_i^*(T_i V_i^*)^* =
\begin{cases}
E_i - E_0 & \text{ if } i=m,\, \,  A(i,m) =1, \\
E_i - E_0 - T_i E_o T_i^* & \text{ if } i=m, \, \, A(i,m) =0, \\
E_i  & \text{ if } i\ne m, \, \, A(i,m) =1, \\
E_i- T_i E_o T_i^*   & \text{ if } i\ne m, \, \, A(i,m) =0.
\end{cases}
\end{equation}
\end{lemma}
Since the partial isometries $V_i, i=1,\dots,N$ on $H_A$ satisfy \eqref{eq:CK},
there exists a unital $*$-monomorphism
$\tau_m: \OA\longrightarrow B(H_A)$
satisfying $\tau_m(S_i) = V_i, i=1,\dots,N$, so that 
$\pi\circ\tau_m: \OA\longrightarrow Q(H_A)$
is a trivial extension.
The above lemma says the following proposition.
\begin{proposition}
For a fixed $m \in \{1,\dots,N\}$, we have
\begin{equation} \label{eq:prop4.3}
d_i(\sigma_{\TA},\tau_m) =
\begin{cases}
-1 & \text{ if } i=m, \, \, A(i,m) =1, \\
-2 & \text{ if } i=m, \, \, A(i,m) =0, \\
0  & \text{ if } i\ne m, \, \, A(i,m) =1, \\
-1 & \text{ if } i\ne m, \, \, A(i,m) =0.
\end{cases}
\end{equation}
\end{proposition}
\begin{proof}
As $H_i = E_iH_A$ and  
\begin{align*}
d_i(\sigma_{\TA},\tau_m) 
=&  \ind_{E_i} T_i V_i^* \\
=&
\dim( \Ker(T_i V_i^*)\text{ in } H_i)  - \dim(\Coker(T_i V_i^*) \text{ in } H_i)\\
=&  - \dim(H_i / T_i V_i^*(T_i V_i^*)^* H_i),
\end{align*}
we get the formula \eqref{eq:prop4.3} by \eqref{eq:lem4.2}. 
\end{proof}
Therefore we have
\begin{theorem}\label{thm:4.4}
Let us denote by $[\TA]_*$ the class in $\Ext_*(\OA)$ of the Toeplitz extension $\sigma_{\TA}$ of $\OA$.
We then  have
\begin{enumerate}
\renewcommand{\theenumi}{(\roman{enumi})}
\renewcommand{\labelenumi}{\textup{\theenumi}}
\item $d_s([{\TA}]_s) = - \hat{\iota}_A(1) - [1_N] $ in $\Z^N/ (I- \widehat{A})\Z^N$,
\item $d_w([{\TA}]_w) =  - [1_N] $ in $\Z^N/ (I- {A})\Z^N$,
\end{enumerate}
where $[1_N]=[(1,\dots,1)]$ means the class of the vector $(1,\dots,1) \in \Z^N$
\end{theorem}
\begin{proof}
Let us denote by $v(m)\in \Z^N$ the column vector in $\Z^N$ 
whose $m$th component is one and the other components are zero's.
Denote by $(1,\dots,1)^t$ the column vector defined by the transpose of the
row vector  whose components are all one's. 
By \eqref{eq:prop4.3}, we have
\begin{align*}
[d_i(\sigma_{\TA},\tau_m)]_{i=1}^N
= & -(1,\dots,1)^t - v(m) +[A(i,m)]_{i=1}^N \\
= & -(I- A)v(m) -(1,\dots,1)^t. 
\end{align*}
Since $[(I-A)v(m)] =\hat{\iota}_A(1)$ in $\Z^N(I-\widehat{A})\Z^N$, 
we have
$
d_s([{\TA}]_s) = - \hat{\iota}_A(1) - [1_N]
$
in
$\Z^N/ (I- \widehat{A})\Z^N.$
As
$\hat{\iota}_A(1)=0$ in $\Z^N/ (I- {A})\Z^N$,
we have 
$
d_w([{\TA}]_w) =  - [1_N]
$
in 
$\Z^N/ (I- {A})\Z^N.
$
\end{proof}
By virtue of the R{\o}rdam's classification theorem for Cuntz--Krieger algebras \cite{Ro}  
(\cite{Cuntz80}, cf. \cite{EFW81})
showing that the $K_0$-group $K_0(\OA)$ with the position of the class $[1]$
of the unit $1$  of $\OA$  in $K_0(\OA)$ 
is a complete invariant of the isomorphism class of the algebra $\OA$, 
we obtain the following corollary. 
\begin{corollary}\label{cor:4.5}
The pair $(\Extw(\OA),[\TA]_w)$ 
of the weak extension group
$\Extw(\OA)$ and  the weak equivalence class
$[\TA]_w$ 
of the Toeplitz extension
$\sigma_\TA$ of the Cuntz--Krieger algebra $\OA$
is a complete invariant of the isomorphism class of the Cuntz--Krieger algebra  
$\mathcal{O}_{A^t}$ for the transposed matrix $A^t$ of the matrix $A$.
This shows that two Cuntz--Krieger algebras $\OA$ and $\OB$ are isomorphic 
if and only if there exists an isomorphism
$\varphi:\Extw(\mathcal{O}_{A^t}) \longrightarrow \Extw(\mathcal{O}_{B^t})$
of groups such that 
$\varphi([{\mathcal{T}_{A^t}}]_w) =[{\mathcal{T}_{B^t}}]_w.$
\end{corollary}
\begin{proof}
As $K_0(\mathcal{O}_{A^t}) \cong \Z^N/(I-A)\Z^N$
and $(\Z^N/(I-A)\Z^N, -[1_N])\cong (\Z^N/(I-A)\Z^N, [1_N])$,
 we have
$$
(\Extw(\OA),[\TA]_w) \cong (\Z^N/(I-A)\Z^N, [1_N])\cong (K_0(\mathcal{O}_{A^t}), [1]).
$$
By virtue of the R{\o}rdam's classification result for Cuntz--Krieger algebras \cite{Ro}
(\cite{Cuntz80}, cf. \cite{EFW81}),
we obtain the desired assertion.
\end{proof}
\begin{remark}
\begin{enumerate}
\renewcommand{\theenumi}{(\roman{enumi})}
\renewcommand{\labelenumi}{\textup{\theenumi}}
\item 
The position $[\TA]_*$ in $\Ext_*(\OA)$
 is not necessarily invariant under the isomorphism class of $\OA$
(see Example 2 in the next section).
\item
The abelian groups $\Extw(\OA)$ and $K_0(\OA)$ are isomorphic,
and two $C^*$-algebras $\OA\otimes K(H)$ and 
$\mathcal{O}_{A^t}\otimes K(H)$ are always isomorphic for every matrix $A$.
There is however an example of an irreducible non permutation matrix $A$ 
such that $\OA$ is not isomorphic to $\mathcal{O}_{A^t}$ 
as in the classification table in \cite{EFW81} of the Cuntz--Krieger algebras for 
$3\times 3$ matrices (see also \cite[Example 2.1]{EFW81}, or Example 4 in the next section).
\end{enumerate}
\end{remark}
\section{Examples}
{\bf 1.} 
Let
$A=
\begin{bmatrix}
1 & \cdots & 1 \\
\vdots& & \vdots \\
1 & \cdots & 1 
\end{bmatrix}
$
be the $N\times N$ matrix with all entries are one's
with $N>1$.
The Cuntz--Krieger algebra $\OA$ is nothing but the Cuntz algebra $\ON$ (see \cite{Cuntz77}).
The element $\iota_A(1)$ in $\Exts(\ON)$ is denoted by $\iota_N(1)$.
 The Toeplitz algebra $\TA$ is also denoted by $\TN$.
As $AR_1 = A$, we have
$ \widehat{A} = A + R_1 - AR_1 = R_1$, so that
$$
I- \widehat{A} = 
\begin{bmatrix}
0     &    -1&    -1&\cdots & -1 \\
0     &     1&     0&\cdots &  0 \\
0     &     0&     1&\ddots &  \vdots \\
\vdots&\vdots&\ddots &\ddots & 0 \\
0&     0&\cdots &0      & 1 \\
\end{bmatrix}.
$$
Define
$$
L_N = 
\begin{bmatrix}
1&     1&     1 &\cdots &  1 \\
0&     1&     0 &\cdots &  0 \\
0&     0&     1 &\ddots &  \vdots \\
\vdots&\vdots&\ddots &\ddots & 0 \\
0&     0&\cdots &0      & 1 \\
\end{bmatrix}
\quad
\text{so that}
\quad
L_N (I-\widehat{A})
=
\begin{bmatrix}
0&     0&     0 &\cdots &  0 \\
0&     1&     0 &\cdots &  0 \\
0&     0&     1 &\ddots &  \vdots \\
\vdots&\vdots&\ddots &\ddots & 0 \\
0&     0&\cdots &0      & 1 \\
\end{bmatrix}.
$$
Hence $L_N$ induces an isomorphism
from $\Z^N / (I- \widehat{A})\Z^N$ to 
$L_N\Z^N / L_N(I - \widehat{A})\Z^N
\cong \Z$ such that 
$$
[v] \in \Z^N / (I - \widehat{A})\Z^N 
\longrightarrow 
[L_N v]\in  L_N\Z^N / L_N(I - \widehat{A})\Z^N
\longrightarrow 
(L_Nv)_1 \in \Z.
$$
For $[v] = \hat{\iota}_N(1) 
=[ (I-A)
\begin{bmatrix}
1\\
0\\
\vdots\\
0
\end{bmatrix}
],
$
we see that 
$$
L_N v = L_N(I-A)\begin{bmatrix}
1\\
0\\
\vdots\\
0
\end{bmatrix}
=
\begin{bmatrix}
1-N\\
0\\
\vdots\\
0
\end{bmatrix}
$$
so that 
$(L_Nv)_1 = 1-N$.
Therefore we have 
$(\Exts(\ON), {\iota}_N(1))\cong (\Z, 1-N)$
and hence the exact sequence
\eqref{eq:thm3.33} goes to
\begin{equation}
\begin{CD}
0 @>>> \Z @>\times(1-N)>> \Z
@>q >> \Z/(1-N)\Z @>>> 0.
\end{CD}
\end{equation}
By using Theorem \ref{thm:4.4}, one may easily compute that  
$$
(\Extw(\ON), [\TN]_w)\cong (\Z/(1-N)\Z, -1), \qquad
(\Exts(\ON), [\TN]_s, \iota_N(1)) \cong (\Z, -1, 1-N). 
$$

{\bf 2.}
Let us denote by $F$  
the Fibonacci matrix 
$
\begin{bmatrix}
1 & 1 \\
1 & 0 
\end{bmatrix}.
$
It is well-known that 
the Cuntz--Krieger algebra 
$\mathcal{O}_F$ is isomorphic to 
the Cuntz algebra $\mathcal{O}_2$.
Hence we have
$\Extw(\mathcal{O}_F)\cong \Extw(\mathcal{O}_2)\cong \{0\}$,
and
$\Exts(\mathcal{O}_F)\cong \Exts(\mathcal{O}_2)\cong \Z.$
By the formula in Theorem \ref{thm:4.4} together with the above Example 1,
we see
$$
(\Exts(\mathcal{O}_F), [\T_F]_s, \iota_F(1)) = (\Z, -2, -1),
\qquad
(\Exts(\mathcal{O}_2), [\T_2]_s, \iota_2(1)) = (\Z, -1, -1).
$$
Hence the position $[\T_F]_s$ in $\Exts(\mathcal{O}_F)$
is different from the position $[\T_2]_s$ in $\Exts(\mathcal{O}_2).$

{\bf 3.}
The weak extension groups 
$\Extw(\mathcal{O}_{A_i}), i=1,2,3,4$ 
of $\mathcal{O}_{A_i}, i=1,2,3,4$
for the following list of matrices $A_i, i=1,2,3,4$ 
have been presented in
\cite[Remark 3.4]{CK}.
Their strong extension groups 
$\Exts(\mathcal{O}_{A_i})$ 
with the positions of the 
element $\iota_{A_i}(1), i=1,2,3,4$
are easily computed  by using Theorem \ref{thm:3.3}.
We also easily know  the positions
$[\T_{A_i}]_*$ in $\Ext_*(\mathcal{O}_{A_i})$ by Theorem \ref{thm:4.4}.
We present the list in the following,  computed without difficulty by hand.
\begin{itemize}
\item 
 $A_1 =
\begin{bmatrix}
0 & 0 & 1\\
1 & 0 &1\\
1 & 1 &1  
\end{bmatrix},
\quad
(\Extw(\mathcal{O}_{A_1}), [\T_{A_1}]_w)
\cong (\Z/3\Z,2),$

\hspace{25mm}
$(\Exts(\mathcal{O}_{A_1}), [\T_{A_1}]_s, \iota_{A_1}(1))
\cong(\Z, 4,  3).
$ 
\item 
 $A_2 =
\begin{bmatrix}
0 & 1 & 1\\
1 & 0 &1\\
1 & 1 &1  
\end{bmatrix},
\quad
(\Extw(\mathcal{O}_{A_2}), [\T_{A_2}]_w)
\cong (\Z/4\Z, 2),$

\hspace{25mm}
$(\Exts(\mathcal{O}_{A_2}), [\T_{A_2}]_s, \iota_{A_2}(1))
\cong(\Z\oplus\Z/2\Z, -2\oplus 0, 2\oplus 1).
$ 
\item 
 $A_3 =
\begin{bmatrix}
0 & 1 & 1\\
1 & 0 &1\\
1 & 1 &0  
\end{bmatrix},
\quad
(\Extw(\mathcal{O}_{A_3}), [\T_{A_3}]_w)
\cong (\Z/2\Z \oplus\Z/2\Z, 0\oplus 0),$

\hspace{25mm}
$(\Exts(\mathcal{O}_{A_3}), [\T_{A_3}]_s, \iota_{A_3}(1))
\cong(\Z\oplus\Z/2\Z\oplus\Z/2\Z, -2\oplus 0\oplus 0, 1\oplus 1\oplus 1).
$ 
\item 
 $A_4 =
\begin{bmatrix}
1 & 0 & 1\\
0 & 1 &1\\
1 & 1 &1  
\end{bmatrix},
\quad
(\Extw(\mathcal{O}_{A_4}), [\T_{A_4}]_w)
\cong (\Z, -1),$

\hspace{25mm}
$(\Exts(\mathcal{O}_{A_4}), [\T_{A_4}]_s, \iota_{A_4}(1))
\cong(\Z\oplus\Z, -2\oplus(-1), 1\oplus 0).
$ 
 \end{itemize}

{\bf 4.}
The matrices
$$
A_5 =
\begin{bmatrix}
1 & 1 & 1\\
1 & 1 &1\\
1 & 0 &0  
\end{bmatrix},
\qquad
A_6 
=A_5^t
=
\begin{bmatrix}
1 & 1 & 1\\
1 & 1 &0\\
1 & 1 &0  
\end{bmatrix}
$$
are examples presented in \cite[Example 2.1]{EFW81}
such that 
$(K_0(\mathcal{O}_{A_5}), [1]) \cong (\Z/2\Z, 1)$
and 
$(K_0(\mathcal{O}_{A_6}), [1]) \cong (\Z/2\Z, 0),$
so that 
$\mathcal{O}_{A_5}$ is not isomorphic to $\mathcal{O}_{A_6}.$
We then see that 
$$
(\Ext_w(\mathcal{O}_{A_5}), [\T_{A_5}]_w) \cong (\Z/2\Z, 0), \qquad
(\Ext_w(\mathcal{O}_{A_6}), [\T_{A_6}]_w) \cong (\Z/2\Z, 1).
$$
We also  easily see that 
\begin{gather*}
(\Exts(\mathcal{O}_{A_5}), [\T_{A_5}]_s, \iota_{A_5}(1))
\cong(\Z, -2,-2), \\
(\Exts(\mathcal{O}_{A_6}), [\T_{A_6}]_s, \iota_{A_6}(1))
\cong(\Z\oplus\Z/2\Z, -1\oplus0, -1\oplus(-1)),
\end{gather*} 
and hence $\Exts(\mathcal{O}_{A_5})$ is not isomorphic to $\Exts(\mathcal{O}_{A_6}).$

\medskip
Some of the results in this paper will be generalized to more general setting 
in a class of $C~*$-algebras associated with symbolic dynamical systems
in \cite{MaPre2021}.

\medskip

{\it Acknowledgment:}
The author would like to thank Joachim Cuntz for his useful comments and suggestions
on a preliminary version of this paper.
This work was supported by JSPS KAKENHI 
Grant Number 19K03537.

\end{document}